\documentclass[11pt]{article}

\usepackage{amssymb}

\usepackage{enumerate}
\usepackage{amsmath}
\usepackage{amssymb,latexsym}
\usepackage{amsthm}
\usepackage{color}
\usepackage{graphicx}
\usepackage{amscd}
\usepackage{hyperref}
\usepackage[normalem]{ulem}
\usepackage{lineno,xcolor}

\newcommand{\F}{{\mathbb F}}

\newcommand{\cK}{{\mathcal K}}

\newcommand{\cF}{{\mathcal F}}

\newcommand{\cH}{{\mathcal H}}

\newcommand{\Lb}{{\mathbb L}}

\newcommand{\cQ}{{\mathcal Q}}

\newcommand{\Fq}{\F_q}

\newcommand{\Fqq}{\F_{q^4}}

\newcommand{\la}{\langle}
\newcommand{\ra}{\rangle}

\newcommand{\lpo}{L_{U(b)}}
\newcommand{\lug}{L_{U(b)}}
\newcommand{\luh}{L_{U(c)}}
\newcommand{\norma}{\N_{q^4/q}}

\DeclareMathOperator{\Tr}{Tr}
\DeclareMathOperator{\N}{N}
\DeclareMathOperator{\PG}{{PG}}
\DeclareMathOperator{\AG}{{AG}}
\DeclareMathOperator{\GL}{{GL}}
\DeclareMathOperator{\GamL}{{\Gamma L}}
\DeclareMathOperator{\PGL}{{PGL}}
\DeclareMathOperator{\PGaL}{P\Gamma L}

\newtheorem{theorem}{Theorem}[section]

\newtheorem{corollary}[theorem]{Corollary}
\newtheorem{proposition}[theorem]{Proposition}

\newtheorem{remark}[theorem]{Remark}
\newtheorem{definition}[theorem]{Definition}

\title{Maximum scattered $\mathbb{F}_q$-linear sets of $\mathrm{PG}(1,q^4)$}
\author{Bence Csajb\'{o}k and Corrado Zanella\thanks{The research was supported
by Ministry for Education, University and Research of Italy MIUR
(Project PRIN 2012 ``Strutture Geometriche, Combinatoria e loro applicazioni'')
and by the Italian National Group for Algebraic and Geometric Structures and their Applications (GNSAGA--INdAM).}}

\begin{document}
\maketitle


%

\begin{abstract}
There are two known families of maximum scattered $\F_q$-linear sets in $\PG(1,q^t)$: the linear sets of pseudoregulus type and for $t\geq 4$ the scattered linear sets found by Lunardon and Polverino. For $t=4$ we show that these are the only maximum scattered $\F_q$-linear sets and we describe the orbits of these linear sets under the groups $\PGL(2,q^4)$ and $\PGaL(2,q^4)$.
\end{abstract}







\section{Introduction}
\label{Intro}

Recent investigations on linear sets in a finite projective line $\PG(1,q^t)$ of rank $t$ concerned:
the hypersurface obtained from the linear sets of pseudoregulus type by applying field reduction \cite{LaShZa2015}; a geometric characterization of the linear sets of pseudoregulus type \cite{CsZa2016a}; a characterization of the clubs, that is, the linear sets of rank $r$
with a point of weight $r-1$ \cite{LaZa2015}; a generalization of clubs in order to construct KM-arcs \cite{KM-arcs};
a condition for the equivalence of two linear sets \cite{CsZa2016,VV2016}; the definition and study of the class of a linear set in order to study their equivalence \cite{CsMaPo2016}; a construction method which yields MRD-codes from maximum scattered linear sets of $\PG(1,q^t)$ \cite{Sh}.
Furthermore, the linear sets in $\PG(1,q^t)$ coincide with the so-called splashes of subgeometries \cite{LaZa2015}.
The results of such investigations make it reasonable to attempt to classify the linear sets in $\PG(1,q^t)$ of rank $t$ for small $t$.

A point in $\PG(1,q^t)$ is the $\F_{q^t}$-span $\la {\bf v}\ra_{\F_{q^t}}$ of a nonzero vector
${\bf v}$ in a two-dimensional vector space, say $W$, over $\F_{q^t}$. If $U$ is a subspace over $\Fq$ of $W$, then
$L_U=\{\la {\bf v}\ra_{\F_{q^t}}\colon {\bf v}\in U\setminus\{ {\bf 0}\}\}$ denotes the associated
\emph{$\Fq$-linear set} (or simply \emph{linear set}) in $\PG(1,q^t)$. The \emph{rank} of such a linear set is $r=\dim_{\Fq}U$.
Any linear set in $\PG(1,q^t)$ of rank greater than $t$ coincides with the whole projective line.
The \emph{weight} of a point $P=\la {\bf v}\ra_{\F_{q^t}}$ is $w(P)=\dim_{\Fq}(U\cap P)$.
If the rank and the size of $L_U$ are $r$ and $(q^r-1)/(q-1)$, respectively, then $L_U$ is \emph{scattered}.
Equivalently, $L_U$ is scattered if and only if all its points have weight one.
A scattered $\Fq$-linear set of rank $t$ in $\PG(1,q^t)$ is \emph{maximum scattered}.
An example of maximum scattered $\Fq$-linear set in $\PG(1,q^t)$ is $L_V$ with $V=\{(u,u^q) \colon u\in\F_{q^t}\}$.
Any subset of $\PG(1,q^t)$ projectively equivalent to this $L_V$ is called \emph{linear set of pseudoregulus type}.
See \cite{CsZa2016a} for a geometric description, and \cite{CsMaPo2016} or the survey \cite{Po2010} for further background on linear sets.
Note that for any $\varphi\in\GamL(2,q^t)$ with related collineation $\tilde\varphi\in\PGaL(2,q^t)$ and any $\Fq$-linear set $L_U$, $L_{U^\varphi}=(L_U)^{\tilde\varphi}$.
In \cite[Theorem 4.5]{CsMaPo2016} it is proved that if $t=4$ and $L_U$ has \emph{maximum field of linearity}
$\Fq$, that is, $L_U$ is not an $\F_{q^s}$-linear set for $s>1$, then any linear set in the same orbit of $L_U$ under the action of $\PGaL(2,q^4)$
is of type $L_{U^\varphi}$ with $\varphi\in\GamL(2,q^4)$. Note that this is not true if $t>4$.
In \cite{LuPo2001}, Lunardon and Polverino construct a class of maximum scattered linear sets:
\begin{theorem}[\cite{LuPo2001}]
\label{t:lpo}
  Let $q$ be a prime power, $t\ge4$ an integer, $b\in\F_{q^t}$
  such that the norm $\N_{q^t/q}(b)$ of $b$ over $\Fq$ is distinct from one, and
  \begin{equation}
  \label{e:ugt}
    U(b,t)=\{(u,bu^q+u^{q^{t-1}}): u\in\F_{q^t}\}.
  \end{equation}
  If $b\neq 0$ then $L_{U(b,t)}$ is a maximum scattered $\Fq$-linear set in $\PG(1,q^t)$ and if $q>3$, then it is not
  of pseudoregulus type.
\end{theorem}
It can be directly seen that $L_{U(0,t)}$ is maximum scattered of pseudoregulus type.
For $t=4$, Theorem \ref{t:lpo} can be extended to $q=3$, as it can be checked
by using the package {\tt FinInG} of {\tt GAP} \cite{Fining1}.
In the following $t=4$ is assumed.
For all $b\in\Fqq$ define
\begin{equation}\label{e:ug}
  U(b)=U(b,4)=\{(x,bx^q+x^{q^3}) : x\in \F_{q^4}\}.
\end{equation}

In section \ref{s:classification} it is shown that $\norma(b)\neq 1$ is a necessary condition to obtain scattered linear sets of $\PG(1,q^4)$ and the case
$\norma(b)=1$ is dealt with. In this case, $\lug$ contains either one or $q+1$ points of weight two, and the remaining points have weight one.

The main result in section \ref{s:cf} is that if $L$ is a maximum scattered
linear set in $\PG(1,q^4)$, then $L$ is projectively equivalent to $\lug$ for some $b\in\Fqq$ with
$\norma(b)\neq1$ (cf. Theorem \ref{t:ess-unico}).

\begin{sloppypar}
In section \ref{s:orbits} the orbits of the $\Fq$-linear sets of rank four in $\PG(1,q^4)$ of type $\lug$, under the actions of
both $\PGL(2,q^4)$ and $\PGaL(2,q^4)$, are completely characterized.
Such orbits only depend on the norm $b^{q^2+1}$ of $b$ over $\F_{q^2}$.
In particular, $\PG(1,q^4)$ contains precisely $q(q-1)/2$ maximum scattered linear sets up to projective equivalence (Theorem \ref{t:orbits}), one of them is of pseudoregulus type, the others are as in Theorem \ref{t:lpo}.
\end{sloppypar}

\section{\texorpdfstring{Classification}{Classification}}
\label{s:classification}

This section is devoted to the classification of all $\lug$ for $b\in\Fqq$,
where $U(b)$ is as in (\ref{e:ug}).

\begin{theorem}\label{c:classes}
For $b\in\Fqq$ the following holds.
\begin{enumerate}
\item If $\N_{q^4/q}(b)\neq 1$, then $\lug$ is scattered.
\item If $\N_{q^4/q^2}(b)= 1$, then $\lug$ has a unique point with weight two, the point $\la(1,0)\ra_{\F_{q^4}}$, and all other with weight one.
\item If $\N_{q^4/q^2}(b)\neq 1$ and $\N_{q^4/q}(b)=1$, then $\lug$ has $q+1$ points with weight two and all other with weight one.
\end{enumerate}
\end{theorem}
\begin{proof}
Put $f_b(x)=bx^q+x^{q^3}$. For $x\in \F_{q^4}^*$ the point $P_x:=\la(x,f_b(x))\ra_{\Fqq}$ of $\lug$ has weight more than one if and only if there exists
$y\in \F_{q^4}^*$ and $\lambda\in \F_{q^4}\setminus \F_q$ such that $\lambda(x,f_b(x))=(y,f_b(y))$.
This holds if and only if $y=\lambda x$ and
\begin{equation}
\label{peso0}
\lambda b x^q+\lambda x^{q^3}-\lambda^qbx^q-\lambda^{q^3}x^{q^3}=0.
\end{equation}
For a given $x$ the solutions in $\lambda$ of \eqref{peso0} form an $\F_q$-subspace whom rank equals to the weight of the point
$P_x$. Since $q$-polynomials over $\F_{q^4}$ of rank 1 are of the form $\alpha\Tr_{q^4/q}(\beta x)\in \F_{q^4}[x]$, it is clear that the kernel of the $\F_q$-linear map in the variable $\lambda$ at the left-hand side of \eqref{peso0} has dimension at most two and hence the weight of each point of $\lug$ is at most two. If $(\lambda,x)$ is a solution of \eqref{peso0} for some $\lambda \in \F_{q^4}$ and $x\in \F_{q^4}^*$, then
$(\lambda',x')$ is also a solution for each $\lambda'\in \la 1, \lambda \ra_{\F_q}$ and $x'\in \la x \ra_{\F_{q^2}}$ and hence for each $\mu\in \F_{q^2}^*$ if $P_x$ has weight two, then $P_{\mu x}:=\la(\mu x,f_b(\mu x)\ra_{\F_{q^4}}$ has weight two as well.
Note that $P_{\mu x}=\la (1,\mu^{q-1}(bx^{q-1}+x^{q^3-1}) )\ra_{\F_{q^4}}$ and hence if $P_x \neq \la(1,0)\ra_{\F_{q^4}}$ has weight two, then $\{P_{\mu x} \colon \mu\in \F_{q^2}^*\}$ is a set of $q+1$ distinct points with weight 2.

The function $f_b(x)$ is not $\F_{q^2}$-linear and hence the maximum field of linearity of $\lug$ is $\F_q$.
It follows (cf. \cite[Proposition 2.2]{CsMaPo2016})) that $\lug$ has at least one point with weight one, say $\la(x_0,f_b(x_0))\ra_{\F_{q^4}}$.
Then the line of $\AG(2,q^4)$ with equation $x_0 Y=f_b(x_0)X$ meets the graph of $f_b(x)$, that is, $\{(x,f_b(x)) \colon x\in \F_{q^4}\}$, in exactly $q$ points. It follows from \cite{B2003,BBBSSz}, see also \cite{Cs2017}, that the number of directions determined by $f_b(x)$ is at least $q^3+1$, and hence also
$|\lug|\geq q^3+1$. Denote by $w_1$ and $w_2$ the number of points of $\lug$ with weight one and two, respectively. Then
\begin{equation}
\label{w1}
w_1+w_2 =|\lug|\geq q^3+1,
\end{equation}
\begin{equation}
\label{w2}
w_1(q-1)+w_2(q^2-1)=q^4-1.
\end{equation}
Subtracting \eqref{w1} $(q-1)$-times from \eqref{w2} gives $w_2(q^2-q)\leq q^3-q$ and hence $w_2 \leq q+1$.
At this point it is clear that in $\lug$ there is either one point with weight two, the point $\la(1,0)\ra_{\F_{q^4}}$, or there are exactly $q+1$ of them
and $\la(1,0)\ra_{\F_{q^4}}$ is not one of them.

If $\N_{q^4/q}(b)\neq 1$, then Theorem \ref{t:lpo} states that $\lug$ is scattered. We show that $\la(1,0)\ra_{q^4}$ has weight two if and only if $\N_{q^4/q^2}(b)=1$. Note that the weight of this point is the dimension of the kernel of $f_b(x)$. If $f_b(x)=0$ for some $x\in \F_{q^4}^*$, then
$b=-x^{q^3-q}$ and hence, by taking $(q^2+1)$-th powers at both sides, $\N_{q^4/q^2}(b)=1$. On the other hand, if $\N_{q^4/q^2}(b)=1$, then
$b=w^{q^2-1}$ for some $w\in \F_{q^4}^*$. Let $\varepsilon$ be a non-zero element of $\F_{q^4}$ such that $\varepsilon^{q^2}+\varepsilon=0$.
Then it is easy to check that the kernel of $f_b(x)$ is $\la (\varepsilon w)^{q^3} \ra_{\F_{q^2}}$ which has dimension two over $\F_q$ and hence
$\la(1,0)\ra_{q^4}$ has weight two.

It remains to prove that if $\N_{q^4/q}(b)=1$ and $\N_{q^4/q^2}(b)\neq1$,
then there is at least one point (hence precisely $q+1$ points) of weight two.
After rearranging in \eqref{peso0}, we obtain
\begin{equation}
\label{peso}
(\lambda-\lambda^q)^{q^3-1}=bx^{q-q^3}.
\end{equation}
By taking $(q^2+1)$-th powers on both sides we can eliminate $x$, obtaining
\begin{equation}
\label{clear0}
(\lambda-\lambda^q)^{(q^3-1)(q^2+1)}=(\lambda-\lambda^q)^{(q-1)(q^2+1)}=b^{q^2+1}.
\end{equation}
It is clear that we can find $\lambda\in \F_{q^4}\setminus \F_q$ satisfying \eqref{clear0} if and only if there exists $\epsilon \in \F_{q^4}^*$ such that
\begin{equation}
\label{eps}
(\lambda-\lambda^q)^{q^3-1}/b = \epsilon^{q^2-1}.
\end{equation}
Then $x\in \la \epsilon^q \ra_{\F_{q^2}}$ with $y=\lambda x$ satisfies our initial conditions in \eqref{peso0}.

Now use $\N_{q^4/q}(b)=1$ and put $b=\mu^{q-1}$ for some $\mu\in \F_{q^4}^*$. Then \eqref{clear0} can be written as
\begin{equation}
\label{clear}
\left(\frac{\lambda-\lambda^q}{\mu}\right)^{(q-1)(q^2+1)}=1.
\end{equation}
We can solve \eqref{clear} if and only if there exists $\delta \in \F_{q^4}^*$ such that
\begin{equation}
\label{eps2}
\left(\frac{\lambda-\lambda^q}{\mu}\right)^{q-1}=\delta^{q^2-1},
\end{equation}
or, equivalently,
\begin{equation}
\label{e:eqgeom}
\left\langle\frac{\lambda-\lambda^q}{\mu}\right\rangle_{\Fq}=\la\delta^{q+1}\ra_{\Fq}.
\end{equation}
Now we will continue in $\PG(\F_{q^4},\F_q)=\PG(3,q)$.
At the left-hand side of \eqref{e:eqgeom} we can see a point of the hyperplane $\cH_{\mu}$ defined as
\[\cH_{\mu}=\{\la z \ra_{\F_q} \colon \Tr_{q^4/q}(\mu z)=0 \},\]
while on the right-hand side we can see a point of the elliptic quadric $\cQ$ defined as
\[\cQ=\{ \la z \ra_{\F_q} \colon z^{(q-1)(q^2+1)}=1 \}.\]
For a proof that $\cQ$ is an elliptic quadric see \cite[Theorem 3.2]{CoSt1995}.
Since $\cQ \cap \cH_{\mu} \neq \emptyset$ it follows that we can always find $\lambda \in \F_{q^4}\setminus \F_q$ satisfying \eqref{eps} and hence $L_{U(b)}$ is not scattered.

%
\end{proof}

\begin{remark}
  The linear sets in Theorem \ref{c:classes} are of sizes $q^3+q^2+q+1$,
  $q^3+q^2+1$, or $q^3+1$.
  The linear set associated with $\{(x,\Tr_{q^4/q}(x)):x\in\Fqq\}$ is of size
  $q^3+1$ as well.
  As it turns out from \cite{BoPo2005} the projective line $\PG(1,q^4)$ also contains
  $\Fq$-linear sets of size $q^3+q^2-q+1$.%
\end{remark}

\section{The canonical form}
\label{s:cf}

In this section $\Lb$ denotes a maximum scattered $\Fq$-linear set
in $\PG(1,q^4)$, not of pseudoregulus type.
In particular, this implies $q>2$.
By \cite{LuPo2004}, $\Lb$ is a projection $p_\ell(\Sigma)$, where the vertex $\ell$ is a line and $\Sigma$
is a $q$-order canonical subgeometry\footnote{Let $\PG(V,\F_{q^t})=\PG(n-1,q^t)$, let $U$ be an $n$-dimensional $\F_q$-vector subspace of $V$,
and $\Sigma=\{ \la {\bf u} \ra_{\F_{q^t}} \colon {\bf u}\in U\setminus \{{\bf 0}\}  \}$.
If $\la \Sigma\ra=\PG(n-1,q^t)$, then $\Sigma$ is a
\emph{($q$-order) canonical subgeometry} of $\PG(n-1,q^t)$.
Here and in the following, angle brackets $\la-\ra$ without a subscript denote projective span in $\PG(n-1,q^t)$, that is, $\PG(3,q^4)$ in our case.}
in $\PG(3,q^4)$, with $\ell\cap\Sigma=\emptyset$.
The axis of the projection is immaterial and can be chosen by convenience.
Let $\sigma$ be a generator of the subgroup of order four of $\PGaL(4,q^4)$
fixing pointwise $\Sigma$.
Let $M$ be a $k$-dimensional subspace of $\PG(3,q^4)$. We say that $M$ is a \emph{subspace of $\Sigma$} if $M\cap \Sigma$ is a $k$-dimensional subpsace of $\Sigma$, which happens exactly when $M^\sigma=M$.


\begin{proposition}\label{p:gamma-fuori}
Let $\Sigma'$ be the unique $q^2$-order canonical subgeometry of $\PG(3,q^4)$
containing $\Sigma$, that is,
the set of all points fixed by $\sigma^2$.
Then the intersection of $\ell$ and $\Sigma'$ is empty.
\end{proposition}
\begin{proof}
  Assume the contrary, that is, there exists a point $P$ in $\ell\cap\Sigma'$.
  Then $P^{\sigma^2}=P$, the subspace $\ell_P=\la P,P^\sigma\ra$
  is a line, and satisfies $\ell_P^\sigma=\ell_P$, whence
  $\ell_P$ is a line of $\Sigma$.
  This implies that $ p_{\ell}(\ell_P)$ is a point, and $\mathbb L$ is not scattered.
\end{proof}

Let $\cK$ and $\cK'$ be the Klein quadrics representing -- via the Pl\"ucker embedding $\wp$ --
the lines of $\Sigma$ and $\Sigma'$.
In order to precisely define $\wp$, take coordinates in $\PG(3,q^4)$ such that $\Sigma$
(resp.\ $\Sigma'$) is the set of all points with coordinates rational over $\Fq$ (resp.\ $\F_{q^2})$, and define the image $r^\wp$ of any line $r$ through minors of order two in the usual way.
Then $\cK=\cK'\cap\PG(5,q)$ by considering $\PG(5,q)$ as a subset of $\PG(5,q^2)$.
The only nontrivial element of the subgroup of $\PGaL(6,q^2)$ fixing $\PG(5,q)$ pointwise is
  \begin{equation}\label{e:tau}
    \tau:\la(x_0,x_1,x_2,x_3,x_4,x_5)\ra_{\F_{q^2}}\mapsto\la(x_0^q,x_1^q,x_2^q,x_3^q,x_4^q,x_5^q)\ra_{\F_{q^2}}.
  \end{equation}
Then $\cK_2^\tau =\cK_2$, and $\sigma\wp=\wp\tau$.

\begin{proposition}\label{l:ellittica}
  Let $S$ be a solid in $\PG(5,q^2)$ such that $(i)\ S\cap\cK'\cong Q^-(3,q^2)$,
  $(ii)\ S\cap\cK=\emptyset$. Then $S\cap S^\tau\cap\cK'$ is a set of two distinct points forming an orbit of $\tau$.
\end{proposition}
\begin{proof}
  If $\dim(S\cap S^\tau)\ge 2$, then $S\cap S^\tau$ contains a plane of $\PG(5,q)$.
  Each plane of $\PG(5,q)$ meets $\cK$ in at least one point of $\PG(5,q)$, contradicting $(ii)$.
  Then $r=S\cap S^\tau$ is a line fixed by $\tau$, so it is a line of $\PG(5,q)$.
  This $r$ is external to the Klein quadric $\cK$ by $(ii)$, hence it meets $\cK'$ in two points.
  Since both of $\cK'$ and $r$ are fixed by $\tau$ the assertion follows.
\end{proof}

\begin{proposition}\label{t:two}
  There is a line $r$ in $\PG(3,q^4)$, such that
  $r$ and $r^\sigma$ are skew lines both meeting $\ell$,
  and $r^{\sigma^2}=r$.
\end{proposition}
\begin{proof}
  Let $\Sigma$ and $\Sigma'$ be as in Proposition \ref{p:gamma-fuori}.
  Since $\ell\cap\Sigma'=\emptyset$, $\ell$ defines a regular (Desarguesian) spread $\cF$ of $\Sigma'$.
  The lines of $\cF$ are all lines $\la P, P^{\sigma^2}\ra\cap\Sigma'$ where $P\in\ell$.
  The image $\cF^\wp$ under the Pl\"ucker embedding of $\cF$ is an elliptic quadric
  $S\cap\cK'\cong Q^-(3,q^2)$ in $\PG(5,q^2)$, $S$ a solid.
  Since $\Lb$ is scattered, there is no line of $\cF$ fixed by $\sigma$, whence
  $S\cap\cK=\emptyset$.
  Then the assertion follows from Proposition \ref{l:ellittica}.
\end{proof}

\begin{theorem}\label{t:ess-unico}
  Any maximum scattered linear $\Fq$-linear set in $\PG(1,q^4)$ is projectively equivalent to
  $\lpo$ for some $b\in\Fqq$, $\norma(b)\neq1$.
\end{theorem}
\begin{proof}
  The set $L_{U(0)}$ is a linear set of pseudoregulus type.
  Now assume that $\Lb=p_\ell(\Sigma)$ is maximum scattered, not of pseudoregulus type.
  Coordinates $X_0,X_1,X_2,X_3$ in $\PG(3,q^4)$ can be chosen such that
  \begin{equation}\label{e:C0}
    \Sigma=\{\la(u,u^q,u^{q^2},u^{q^3})\ra_{\Fqq}: u\in\Fqq^*\},
  \end{equation}
  and a generator of the subgroup of $\PGaL(4,q^4)$ fixing $\Sigma$ pointwise is
  \begin{equation}\label{e:sigma}
    \sigma:\la(x_0,x_1,x_2,x_3)\ra_{\Fqq}\mapsto
    \la(x_3^q,x_0^q,x_1^q,x_2^q)\ra_{\Fqq}.
  \end{equation}
  Define $C=\ell\cap r$, where $r$ is as in
  Proposition \ref{t:two}. The points $C$ and $C^{\sigma^2}$ lie on $r$, as well as the points
  $C^\sigma$ and $C^{\sigma^3}$ lie on $r^\sigma$.
  By Proposition \ref{p:gamma-fuori}, $C\neq C^{\sigma^2}$ and $C^\sigma\neq C^{\sigma^3}$.
  This implies $\ell\subset\la C,C^\sigma,C^{\sigma^3}\ra$, and $\la C,C^\sigma,C^{\sigma^2},C^{\sigma^3}\ra=\PG(3,q^4)$.
  Since the stabilizer of $\Sigma$ in $\PGL(4,q^4)$ acts transitively on the
  points $C$ of $\PG(3,q^4)$ such that
  $\la C,C^\sigma,C^{\sigma^2},C^{\sigma^3}\ra=\PG(3,q^4)$
  \cite[Proposition 3.1]{BoPo2005}, it may be assumed that $C=\la(0,0,1,0)\ra_{\Fqq}$, whence
  \[
    \ell=\la(0,0,1,0),(0,a,0,-b)\ra_{\Fqq},
  \]
  for some $a,b\in\Fqq$, not both of them zero. If $a=0$, then $\Lb$ is of pseudoregulus type
  \cite[Theorem 2.3]{CsZa2016a}, so $a=1$ may be assumed.
  For any point $P_u=\la(u,u^q,u^{q^2},u^{q^3})\ra_{\Fqq}$ in $\Sigma$, the plane
  containing $\ell$ and $P_u$ has coordinates $[u^{q^3}+bu^q,-bu,0,-u]$, and this leads to the desired form
  for the coordinates of $\Lb$.
\end{proof}

\section{Orbits}\label{s:orbits}

Analogously to the definition of the $\GamL$-class of linear sets (cf.\ Definition 2.4 in \cite{CsMaPo2016}) we define the $\GL$-class, which will be needed to study $\PGL(2,q^4)$-equivalence. Note that for any scattered $\Fq$-linear set the maximum field of linearity is $\Fq$.

\begin{definition}
  Let $L_U$ be an $\Fq$-linear set of $\PG(1,q^t)$ of rank $t$ with maximum field of linearity $\Fq$.  We say that $L_U$ is of \emph{$\GamL$-class} $s$ [resp.\ \emph{$\GL$-class} $s$] if $s$ is the largest integer such that
  there exist $\Fq$-subspaces $U_1$, $U_2$, $\ldots$, $U_s$ of $\F_{q^t}^2$ with $L_{U_i}=L_U$
  for $i\in\{1,2,\ldots,s\}$ and there is no $\varphi\in\GamL(2,q^t)$ [resp.\ $\varphi\in\GL(2,q^t)$]
  such that $U_i=U_j^\varphi$ for each $i\neq j$, $i,j\in\{1,2,\ldots,s\}$.
\end{definition}

The first part of the following result is \cite[Theorem 4.5]{CsMaPo2016}, while the second part follows from its proof.
We briefly summarize the main steps of the proof from \cite{CsMaPo2016}.

\begin{sloppypar}
\begin{theorem}{\cite[Theorem 4.5]{CsMaPo2016}}
\label{CMP45}
  Each $\Fq$-linear set of rank four in $\PG(1,q^4)$, with maximum field of linearity $\Fq$, is of $\GamL$-class one.
  More precisely, if $L_U=L_V$ for some 4 dimensional $\Fq$-subspaces $U$, $V$ of $\Fqq^2$, then there exists
  $\varphi\in\GamL(2,q^4)$ such that $U^\varphi=V$. Also, $\varphi$ can be chosen such that it has companion automorphism either the identity, or $x \mapsto x^{q^2}$.
\end{theorem}
\end{sloppypar}
\begin{proof}
  Assume $L_U=L_V$.
  We may assume $\la(0,1)\ra_{\F_{q^4}} \notin L_U$. Then
  $U=U_f=\{(x,f(x))\colon x\in \F_{q^4}\}$ and $V=V_g=\{(x,g(x))\colon x\in \F_{q^4}\}$ for some $q$-polynomials $f$ and $g$ over $\F_{q^4}$.
  By \cite[Proposition 4.2]{CsMaPo2016}, either $g(x)=f(\lambda x)/\lambda$, or $g(x)={\hat f}(\lambda x)/\lambda$ for some $\lambda\in \F_{q^4}^*$, where
  here $\hat f$ denotes the adjoint map of $f$ with respect to the bilinear form $< x,y >:=\Tr_{q^4/q}(xy)$.
  The $\F_{q^4}$-linear map ${\bf v} \mapsto \lambda{\bf v}$ maps $U_g$ to one of $U_f$, or $U_{\hat f}$.
  In the proof of \cite[Theorem 4.5]{CsMaPo2016}, a $\kappa\in\GamL(2,q^4)$ with companion
  automorphism the identity, or $x\mapsto x^{q^2}$ is determined such that $U_f^\kappa=U_{\hat f}$.
\end{proof}

\begin{theorem}\label{GL-one}
  For any $b\in\Fqq$, $\lug$ is of $\GL$-class one.
\end{theorem}
\begin{proof}
  By Theorem \ref{CMP45}, if $L_{U(b)}=L_V$, then there exists $\varphi\in\GamL(2,q^4)$ such that $U(b)^\varphi=V$ and
  the companion automorphism of $\varphi$ is $x \mapsto x^{q^2}$, or the identity. In order to prove the statement it is enough to show that
  $U(b)$  and $U(b)^{q^2}=\{(x^{q^2},y^{q^2})\colon (x,y)\in U(b)\}$ lie on the same orbit of $\GL(2,q^4)$.
  If $b=0$, then $U(b)=U(b)^{q^2}$. If $b\neq 0$, then for any $u\in\Fqq$,
  \[
    \begin{pmatrix}b^{q^3}&0\\ 0&b^{q^2}\end{pmatrix}
    \begin{pmatrix}u\\ bu^q+u^{q^3}\end{pmatrix}
    =\begin{pmatrix} b^qu^{q^2}\\ b\left(b^qu^{q^2}\right)^q+\left(b^qu^{q^2}\right)^{q^3}\end{pmatrix}^{q^2}
    =\begin{pmatrix}v\\ bv^q+v^{q^3}\end{pmatrix}^{q^2},
  \]
  with $v=b^qu^{q^2}$.
\end{proof}

\begin{corollary}\label{c:GL}
  Let $b,c\in\Fqq$.
  The linear sets $\lug$ and $\luh$ are projectively equivalent if and only if $U(b)$ and ${U(c)}$ are
  in the same orbit under the action of $\GL(2,q^4)$.
\end{corollary}

\begin{proof}
  The ``if'' part is obvious, so assume that $\lug^{\tilde\kappa}=\luh$ where $\kappa\in\GL(2,q^4)$.
  Then $L_{U(b)^\kappa}=\luh$ and by Theorem \ref{GL-one} there is $\kappa'\in\GL(2,q^4)$  such that $U(b)^{\kappa\kappa'}={U(c)}$.
\end{proof}

It follows that in order to classify the $\Fq$-linear sets $L_{U(b)}$ up to $\PGL(2,q^4)$ and $\PGaL(2,q^4)$-equivalence, it is enough to determine the orbits of the subspaces $U(b)$ under the actions of $\GamL(2,q^4)$ and $\GL(2,q^4)$.

\begin{theorem}
\label{t:orbits}
 Let $q$ be a power of a prime $p$.
  \begin{enumerate}[(i)]
  \item For any $b,c\in\Fqq$, $\lug$ and $\luh$ are equivalent up to an element of $\PGaL(2,q^4)$
  if and only if $c^{q^2+1}=b^{\pm p^s(q^2+1)}$ for some integer $s\ge0$.
  \item For any $b,c\in\Fqq$, the linear sets  $\lug$ and $\luh$ are projectively equivalent
  if and only if $c^{q^2+1}=b^{q^2+1}$ or $c^{q^2+1}=b^{-q(q^2+1)}$.
  \item All linear sets described in 2. of Theorem \ref{c:classes} are projectively equivalent.
  \item There are precisely $q(q-1)/2$ distinct linear sets up to projective equivalence in the family described in 1. of Theorem \ref{c:classes},
  and these are the only maximum scattered linear sets of $\PG(1,q^4)$.
  \item There are precisely $q$ distinct linear sets up to projective equivalence in the family described in 3. of Theorem \ref{c:classes}.
  \end{enumerate}
\end{theorem}
\begin{proof}
Take $b\in \F_{q^4}^*$. If $L_{U(b)}$ is not scattered, then it clearly cannot be equivalent to $L_{U(0)}$ (the scattered linear set of pseudoregulus type), while if $L_{U(b)}$ is scattered, then it follows from Theorem \ref{t:lpo} (and from a computer search when $q=3$) that $U(b)$ and $U(0)$ yield projectively inequivalent linear sets. Since the automorphic collineations $(x,y)\mapsto (x^{p^s},y^{p^s})$ fix $U(0)$, it also follows that $L_{U(0)}$ and $L_{U(b)}$ lie on different orbits of $\PGaL(2,q^4)$. Thus (i) and (ii) are true when one of $b$ or $c$ is zero, so from now on we may assume $b\neq 0$ and $c\neq 0$.

The sets $\lug$ and $\luh$ are equivalent up to elements of $\PGaL(2,q^4)$ if and only for some $\psi=p^k$, $k\in\mathbb N$
and some $A,B,C,D\in\Fqq$ such that $AD-BC\neq0$ the following holds:
 \begin{equation}
 \label{e:matriciale}
 \left\{\begin{pmatrix}
 A&B \\
 C&D
 \end{pmatrix}
\begin{pmatrix}
u^\psi\\
b^\psi u^{\psi q}+u^{\psi q^3}
\end{pmatrix} \colon u\in \F_{q^4}\right\}=
\left\{\begin{pmatrix}
v\\
cv^q+v^{q^3}
\end{pmatrix} \colon v\in \F_{q^4} \right\}.
\end{equation}
    Furthermore, by Corollary \ref{c:GL}, $\lug$ and $\luh$ are projectively equivalent if, and only if,
  (\ref{e:matriciale}) has a solution with $\psi=1$.
  This leads to a polynomial in $u^\psi$ of degree at most $q^3$ which is identically zero.
  Equating its coefficients to zero,
  \begin{equation}\label{e:sistema}
    \left\{\begin{array}{rcl}
      A^{q^3}-D&=&0\\ B^qb^{\psi q}c+B^{q^3}&=&0\\ A^qc-Db^\psi&=&0\\
      B^qc+B^{q^3}b^{\psi q^3}-C&=&0.
    \end{array}\right.
  \end{equation}

  Assume that $\lug$ and $\luh$ are in the same orbit of $\PGaL(2,q^4)$, and take $\psi=1$ in case
  they are also projectively equivalent.
  If $D\neq0$, then the first and third equations imply $b^\psi=D^{q^2-1}c$ and so $c^{q^2+1}=b^{\psi(q^2+1)}$.
  If $D=0$, then $BC\neq0$; from the second equation, $(b^{\psi q}c)^{q^2+1}=1$, hence
  $c^{q^2+1}=b^{-\psi q(q^2+1)}$. This proves the only if parts of (i) and (ii).

  Conversely, if $c^{q^2+1}=b^{p^s(q^2+1)}$ for some $s\in \mathbb{N}$,
  then $b^{p^s} c^{-1}=\delta^{q^2-1}$ for some $\delta\in\Fqq^*$.
  The quadruple $A=\delta^q$, $B=C=0$, $D=\delta$ with $\psi=p^s$ is a solution of (\ref{e:sistema}) with $AD-BC\neq 0$.
  This proves the if part of (i) when $c^{q^2+1}=b^{p^s(q^2+1)}$ and the if part of (ii) when $c^{q^2+1}=b^{q^2+1}$.
  If $b^{q^2+1}=c^{q^2+1}=1$, i.e. when $U(b)$ and $U(c)$ define linear sets described in 2. of Theorem \ref{c:classes}, then the above condition holds, thus
  (iii) follows. From now on we may assume $b^{q^2+1}\neq 1$ and $c^{q^2+1}\neq 1$.

  Assume $c^{q^2+1}=b^{-p^s(q^2+1)}$ for some $s\in\mathbb N$, i.e.\ $b^{p^s}c=\varepsilon^{q^2-1}$ for some $\varepsilon\in\Fqq^*$.
  Define $\psi=p^sq^3$. A $\rho\in\Fqq^*$ exists such that $\rho^{q^2-1}=-1$.
  Take $A=D=0$, $B=(\rho\varepsilon)^{q^3}$, $C=\varepsilon\rho c(1-b^{p^s(q^2+1)})$.
  If $C=0$, then $b^{q^2+1}=1$, a contradiction.
  So $AD-BC\neq0$ and (\ref{e:sistema}) has a solution.
  If $p^s=q$, then $\psi=1$, hence in this case $\lug$ and $\luh$ are projectively equivalent.
  This finishes the proofs of (i) and (ii).

  Now we prove (iv). Note that $\norma(b)=(b^{q^2+1})^{q+1}$ for any $b\in\Fq$,
  therefore, $\lug$ is a maximum scattered $\Fq$-linear set not of pseudoregulus type
  if, and only if, $b^{q^2+1}$ is an element of the set
  \[
    S=\{x\in\mathbb F_{q^2}^*\colon x^{q+1}\neq1\}.
  \]
  The orbits of point sets of type $\lug$, $b\neq0$, under the action of $\PGL(2,q^4)$
  are as many as the pairs $\{x,x^{-q}\}$ of elements in $S$. Since all such pairs are made of distinct elements,
  adding one for the linear set of pseudoregulus type, one obtains
  \[
    1+\frac{q^2-q-2}{2}=\frac{q(q-1)}{2}.
  \]
  Finally we prove (v). $\lug$ is an $\Fq$-linear set described in 3. of Theorem \ref{c:classes}
  if, and only if, $b^{q^2+1}$ is an element of the set
  \[
    Z=\{x\in\mathbb F_{q^2}\setminus\{1\} \colon x^{q+1}=1\}.
  \]
  The orbits of point sets of this type under the action of $\PGL(2,q^4)$ are as many as the pairs $\{x,x^{-q}\}$ of elements in $Z$.
  Since for each $x\in Z$ we have $x=x^{-q}$, this number is $q$.
  \end{proof}

\begin{remark}
  The number of orbits of maximum scattered linear sets under the action of $\PGaL(2,q^4)$ depends on the exponent $e$ in $q=p^e$.
  A general formula is not provided here.
  For $e=1$ each orbit which does not arise from the linear set of pseudoregulus type is related to two or four norms over $\F_{q^2}$,
  according to whether $\N_{q^4/q^2}(b)\in\Fq\setminus\{0,1,-1\}$ or not.
  This leads  (including now the linear set of pseudoregulus type) to a total number of $(q^2-1)/4$ orbits for odd $q$.
  \end{remark}

\section*{Acknowledgement} The authors of this paper thank Michel Lavrauw,
Giuseppe Marino and
Olga Polverino for useful discussions and suggestions during the development
of this research.





\bigskip

\noindent Bence Csajb\'ok\\
MTA--ELTE Geometric and Algebraic Combinatorics Research Group,\\
E\"otv\"os Lor\'and University,\\
 H--\,1117 Budapest, P\'azm\'any P\'eter S\'et\'any 1/C, Hungary\\
{\em csajbok.bence@gmail.com}

\bigskip

\noindent Corrado Zanella\\
Dipartimento di Tecnica e Gestione dei Sistemi Industriali,\\
Universit\`a di Padova, Stradella S. Nicola, 3, I-36100 Vicenza, Italy\\
{\em corrado.zanella@unipd.it}

\end{document}